\documentclass{article}
\usepackage{spconf,graphicx}

\include{cmd_latex}

\title{Iteration-Complexity of a Generalized Forward Backward Splitting Algorithm}
\name{Jingwei Liang$^{1}$, Jalal M. Fadili$^{1}$ and Gabriel Peyr\'{e}$^{2}$\thanks{This work has been supported by the ERC project SIGMA-Vision and l'Institut Universitaire de France. We would like to thank Yuchao Tang for pointing \cite{ogura2002non} to us.}}
\address{$^{1}$ GREYC, CNRS-ENSICAEN ~~~~ $^{2}$ CEREMADE, CNRS-Paris-Dauphine}
\begin{document}
%
\maketitle
\begin{abstract}
In this paper, we analyze the iteration-complexity of Generalized Forward--Backward (GFB) splitting algorithm, as proposed in~\cite{gfb2011}, for minimizing a large class of composite objectives $f + \sum_{i=1}^n h_i$ on a Hilbert space, where $f$ has a Lipschitz-continuous gradient and the $h_i$'s are simple (\ie~their proximity operators are easy to compute). We derive iteration-complexity bounds (pointwise and ergodic) for the inexact version of GFB to obtain an approximate solution based on an easily verifiable termination criterion. Along the way, we prove complexity bounds for relaxed and inexact fixed point iterations built from composition of nonexpansive averaged operators. These results apply more generally to GFB when used to find a zero of a sum of $n > 0$ maximal monotone operators and a co-coercive operator on a Hilbert space. The theoretical findings are exemplified with experiments on video processing.
\end{abstract}


\section{Introduction}
\label{sec:intro}

\subsection{Problem statement}
Many structured convex optimization problems in science and engineering, including signal/image processing and machine learning, can be cast as solving
\[
\label{eq:minPhi}
\min_{x\in\H} J(x) := f(x)+\sum_{i=1}^n h_i(x),
\]
where $f\in\lsc$ has ${\beta}^{-1}$-Lipschitz continuous gradient, $h_i\in\lsc$ is simple, and $\lsc$ is the class of lower semicontinuous, proper, convex functions from a Hilbert space $\H$ to $]-\infty,+\infty]$. Some instances of \eqref{eq:minPhi} in signal, image and video processing are considered in Section~\ref{sct:num_expmt} as illustrative examples.

Assume that $\AM J \neq \emptyset$ and that the qualification condition 
\begin{equation*}
\(0,\ldots,0\)\in\mathrm{sri}\ll\(x-y_1,\ldots,x-y_n\)|x\in\H,~\forall i, y_i \in \dom h_i\rr
\end{equation*}
holds, where $\mathrm{sri}$ is the strong relative interior, see~\cite{bauschke2011convex}. Thus, minimizing $J$ in \eqref{eq:minPhi} is equivalent to
\[
\label{eq:mnticln}
\mathrm{Find}~x\in \mathrm{zer}(\partial J):=\left\{ x\in\H|0\in \nabla f(x)+\msum_{i=1}^n \partial h_i(x) \right\} .
\]
Although we only focus on optimization problems \eqref{eq:minPhi} in the sequel, our results apply more generally to monotone inclusion problems of the form
\begin{equation}
\label{eq:mnticlngen}
\mathrm{Find}~x\in \left\{ \mathrm{zer}(B+\sum_{i=1}^n A_i):=\big\{ x\in\H|0\in Bx+\sum_{i=1}^n A_ix \big\} \right\},
\end{equation}
where $B:\H\mapsto\H$ is $\beta$-cocoercive, and $A_i:\H\mapsto 2^\H$ is a maximal monotone set-valued map.

In this paper, we will establish iteration-complexity bounds of the inexact GFB algorithm~\cite{gfb2011} for solving \eqref{eq:mnticln}, whose steps are summarised in Algorithm~\ref{alg:GFB}. There, $\varepsilon_1^k$ and $\varepsilon_{2,i}^k$ are the errors when computing $\nabla f(\cdot)$ and $\prox_{\frac{\gamma}{\omega_i}h_i}(\cdot)$.
\begin{algorithm}
\label{alg:GFB}
\caption{Inexact GFB Algorithm for solving \eqref{eq:mnticln}.}
\KwIn{$\(z_i\)_{i\in\{1,\cdots,n\}}$, $\(\omega_i\)_{i\in\{1,\cdots,n\}}~\mathrm{and}~\sum_{i=1}^{n}\omega_i=1$, $\gamma\in]0,2\beta[$, $\alpha = \frac{2\beta}{4\beta-\gamma}$ and $\lambda_k\in]0, \tfrac{1}{\alpha}[$.}
$k = 0,~x^0 = \sum_{i=1}^n\omega_i z_i^0$\;
\Repeat{convergence}{
\For{$i=1,\ldots,n$}{
  $v_i^{k+1} = \prox_{\frac{\gamma}{\omega_i}h_i} \big( 2x^k-z_i^k-\gamma \nabla f(x^k) + \varepsilon_1^k \big) + \varepsilon_{2,i}^k;$ \\
  $z_i^{k+1} = z_i^k+\lambda_k ( v_i^{k+1}-x^k );$
}
$x^{k+1} = \sum_{i=1}^n\omega_i z_i^{k+1}$\;
$k = k+1$\;
}
\Return{$x$}\;
\end{algorithm}

When $n=1$, GFB recovers the Forward--Backward splitting algorithm \cite{combettes2005signal}, and when $\nabla f=0$, GFB specializes to the Douglas--Rachford algorithm on product space~\cite{lions1979splitting}.

There has been a recent wave of interest in splitting algorithms to solve monotone inclusions taking the form of \eqref{eq:mnticln} or \eqref{eq:mnticlngen}, or even more general. In particular, several primal-dual splitting schemes were designed such as those in \cite{combettes2012primal,vu2011splitting} or \cite{condat2012primal} in the context of convex optimization. See also \cite{bot2012convergence,bot2013convergence} for convergence rates analysis. The authors in \cite{monteiro2010complexity,monteiro2011complexity} analyze the iteration-complexity of the hybrid proximal extragradient (HPE) method proposed by Solodov and Svaiter. It can be shown that the GFB can be cast in the HPE framework but only for the exact and unrelaxed (\ie $\lambda_k=1$) case. 


\subsection{Contributions}
In this paper, we establish pointwise and ergodic iteration-complexity bounds for sequences generated by inexact and relaxed fixed point iterations, in which, the fixed point operator is $\alpha$-averaged. It is a generalization of the result of \cite{cominetti2012rate} to the inexact case, and of \cite{he2011DRconvergence} who only considered the exact Douglas--Rachford method. Then we apply these results to derive iteration-complexity bounds for the GFB algorithm to solve \eqref{eq:minPhi}. This allows us to show that $O(1/\epsilon)$ iterations are needed to find a pair $((u_i)_{i},g)$ with the termination criterion $\norm{g+\nabla f (\ssum_i\omega_i u_i)}^2 \leq \epsilon$, where $g \in \sum_i \partial_i h_i(u_i)$. This termination criterion can be viewed as a generalization of the classical one based on the norm of the gradient for the gradient descent method. The iteration-complexity improves to $O(1/\sqrt{\epsilon})$ in ergodic sense for the same termination criterion. 

\section{Iteration-complexity bounds}
\label{sct:itr_cmplx_bnd}

\subsection{Preliminaries}
\label{sct:prlim}

The class of $\alpha$-averaged non-expansive operators, $\alpha \in ]0,1[$, is denoted $\A(\alpha)=\{T: T = \Id + \alpha(R - \Id)\}$ for some non-expansive operator $R$. For obvious space limitations, we recall in Section~\ref{sec:proofs} only properties of these operators that are essential to our exposition. The reader may refer to e.g., \cite{bauschke2011convex} for a comprehensive account. 


Let $(\omega_i)_{i\in\{1,\ldots,n\}}\in]0,1]^n~\st~\sum_i \omega_i=1$. Consider the product space $\bH:=\H^n$ endowed with scalar product $\bdprod{\cdot}{\cdot}$
\begin{equation*}
\forall \bx=(x_i)_i, \by=(y_i)_i \in \bH,~\bdprod{\bx}{\by} ={\msum}_i \omega_i\<x_i,y_i\>,
\end{equation*}
and the corresponding norm $\bnorm{\cdot}$. Define the non-empty subspace $\bS\subset\bH:=\{\bx=\(x_i\)_i\in\bH|x_1=\ldots=x_n\}$, and its orthogonal complement $\bS^\perp\subset\bH:=\{\bx=\(x_i\)_i\in\bH | \sum_{i=1}^n \omega_ix_i=0\}$. Denote $\bId$ as the identity operator on $\bH$, and the canonical isometry: $\bC:\H\mapsto\bS, x\mapsto(x,\ldots,x)$.

\subsection{Inexact relaxed fixed point equation of GFB}
Denote $\bP_{\bS}: \bH\rightarrow\bH, \bz\mapsto\bC(\ssum_i\omega_iz_i)$, $\bR_{\bS}=2\bP_{\bS}-\bId$, $\bB : \bH \to \bH,~ \bx = (x_i)_i \mapsto (\nabla f(x_i))_i$, $\bJ_{\bgamma\cdot\bA}=(\prox_{\frac{\gamma}{\omega_i}h_i})_i$, and $\bR_{\bgamma\cdot\bA}=2\bJ_{\bgamma\cdot\bA}-\bId$. 

Let $\bT_{1,\bgamma} = \frac{1}{2}[\bR_{\bgamma\cdot\bA}\bR_{\bS} + \bId]$ and $\bT_{2,\gamma} = [\bId - \gamma\bB \bP_{\bS}]$. We can now define the inexact version of GFB.
\begin{proposition}
\label{T_prop}
\begin{enumerate}[label={\rm (\roman{*})}, ref={\rm \ref{T_prop} (\alph{*})}, leftmargin=0cm,itemindent=0.5cm,labelwidth=\itemindent,labelsep=0cm,align=left]
\item \label{T_prop1} The composed operator $\bT_{1,\bgamma}\circ\bT_{2,\gamma}$ is $\alpha$-averaged monotone with $\alpha = \frac{2\beta}{4\beta-\gamma}$;
\item \label{T_prop3} The inexact GFB is equivalent to the following relaxed fixed point iteration
\[
\label{eq:error_fp}
\bz^{k+1} = \bz^{k} + \lambda_k \big( \bT_{1,\bgamma} (\bT_{2,\gamma}\bz^k+\bepsilon_1^k ) + \bepsilon_{2}^k - \bz^k \big) ~,
\]
and $(\bz^k)_{k\in\N}$ is quasi-Fej\'er monotone with respect to $\Fix(\bT_{1,\bgamma}\circ\bT_{2,\gamma}) \neq \emptyset$.
\end{enumerate}

\end{proposition}
\begin{proof}
(\rmnum{1}) This a consequence of \cite[Proposition 4.12-13]{gfb2011} and \cite[Theorem 3]{ogura2002non}. (\rmnum{2}) See \cite[Theorem 4.17]{gfb2011}, and \cite[Theorem 3.1]{combettes2004solving} since $\AM J \neq \emptyset$.
\end{proof}

To further lighten the notation, let $\bT = \bT_{1,\bgamma} \circ \bT_{2,\gamma}$. Then \eqref{eq:error_fp} can be rewritten as
\begin{equation*}
\label{eq:error_fp_T}
\bz^{k+1} = \bT_{k} \bz^{k} + \lambda_{k}\bepsilon^{k},
\end{equation*}
where $\bT_{k}=\lambda_{k}\bT + (1-\lambda_{k})\bId \in \A(\alpha\lambda_k)$, $\bepsilon^k = \bT_{1,\bgamma} (\bT_{2,\gamma}\bz^k+\bepsilon_1^k ) + \bepsilon_{2}^k - \bT\bz^k$. We now define the residual term that will be used as a termination criterion for \eqref{eq:error_fp}, \ie
\begin{equation}
\label{eq:ek}
\be^{k} = (\bId-\bT)\bz^{k}={(\bz^{k}-\bz^{k+1})}/{\lambda_{k}} + \bepsilon^{k}.
\end{equation}


\subsection{Iteration complexity bounds of \eqref{eq:error_fp}}
We are now in position to establish our main results on pointwise and ergodic iteration-complexity bounds for the inexact relaxed fixed point iteration \eqref{eq:error_fp}. The proofs are deferred to Section~\ref{sec:proofs}. Define $\tau_k = \lambda_k(\frac{1}{\alpha} - \lambda_k)$, $\underline{\tau}=\inf_{k\in\N}\tau_k$, $\overline{\tau}=\sup_{k\in\N}\tau_k$.
Let $d_0 = \bnorm{\bz^0-\bz^\star}$ be the distance from $\bz^0$ to $\bz^\star \in \Fix \bT$,
$\nu_1 = 2\sup_{k\in\N}\bnorm{\bT_{\lambda_k}\bz^{k}-\bz^\star} + \sup_{k\in\N}\lambda_k\bnorm{\bepsilon^{k}}$ and $\nu_2 = 2\sup_{k\in\N}\bnorm{\be^k-\be^{k+1}}$.
Let $\ell^1_+$ denote the set of summable sequences in $[0,+\infty[$. 


\begin{theorem}[Pointwise iteration-complexity bound of \eqref{eq:error_fp}] 
\label{thm:pointwise_irfpi_bounds} 
$\\$ $\vspace{-0.5cm}$
\begin{enumerate}[label={\rm (\roman{*})}, ref={\rm \ref{thm:pointwise_irfpi_bounds} (\roman{*})}, leftmargin=0cm,itemindent=0.5cm,labelwidth=\itemindent,labelsep=0cm,align=left]
\item If
\beq
\label{condition2} 
\lambda_k\in]0, 1/\alpha[, ~~ (\tau_k)_{k\in\N} \notin \ell^1_+ ~~\mathrm{and}~~ (\lambda_k\bnorm{\bepsilon^k})_{k\in\N} \in \ell^1_+ ~, 
\eeq
then the sequence $(\be^k)_{k\in\N}$ converges strongly to $0$, and $(\bz^k)_{k\in\N}$ converges weakly to a point $\bz^\star \in \Fix(\bT)$. 
\item \label{pointwise_bnds2} If
\beq
\label{condition3}
0 < \inf_{k \in \N} \lambda_k \leq \sup_{k \in \N} \lambda_k < \tfrac{1}{\alpha}~~\mathrm{and}~~\big( (k+1)\bnorm{\bepsilon^k} \big)_{k\in\N} \in \ell^1_+ ~,
\eeq
then
$C_1 = \nu_1\sum_{j \in \N}\lambda_j\bnorm{\bepsilon^j} + \nu_2\overline{\tau}\sum_{\ell \in \N}(\ell+1)\bnorm{\bepsilon^{\ell}} < +\infty$, and 
\[
\label{eq:bound1}
\bnorm{\be^k} \leq \sqrt{\frac{d_0^2 + C_1}{\underline{\tau}(k+1)}};
\]
\item \label{pointwise_bnds3} If $\frac{1}{2\alpha} \leq \lambda_k \leq \sup_{k \in \N} \lambda_k < \frac{1}{\alpha}$ is non-decreasing, then 
\[
\label{eq:bound2}
\bnorm{\be^k} \leq \sqrt{\frac{d_0^2 + C_2}{\tau_k(k+1)}}.
\]
where $C_2 = \nu_1\sum_{j \in \N}\lambda_j\bnorm{\bepsilon^j} + \nu_2\tau_0\sum_{\ell \in \N}(\ell+1)\bnorm{\bepsilon^{\ell}} < +\infty$.
\end{enumerate}
\end{theorem}

In a nutshell, after $k \geq O\big((d_0^2+C_2)/{\epsilon}\big)$ iterations, \eqref{eq:error_fp} achieves the termination criterion $\bnorm{\be^{k}}^2 \leq \epsilon$. 

\newpage

Denote now $\Lambda_k=\sum_{j=0}^k\lambda_{j}$, and define $\bar{\be}^k=\frac{1}{\Lambda_k}\sum_{j=0}^k\lambda_{j}\be^j$. 
We have the following theorem.

\begin{theorem}[Ergodic iteration-complexity bound of \eqref{eq:error_fp}]
\label{thm:ergodic_irfpi_bound}
If $\lambda_k \in ]0, 1[$ and $C_3 = \sum_{j=0}^{+\infty}\lambda_j\bnorm{\bepsilon^j} < +\infty$, then
\begin{equation*}
\bnorm{\bar{\be}^k} \leq 2\big(d_0 + C_3\big)/{\Lambda_k}.
\end{equation*}
\end{theorem}
If $\inf_k \lambda_k > 0$, then we get the iteration-complexity $O(1/\sqrt{\epsilon})$ in ergodic sense for \eqref{eq:error_fp}.


\subsection{Iteration complexity bounds of \eqref{eq:mnticln}}

We now turn to the complexity bounds of the GFB applied to solve \eqref{eq:mnticln} (or equivalently \eqref{eq:minPhi}). 

From the quantities used in Algorithm \ref{alg:GFB}, let's denote $\bu^{k+1}=(u_i^{k+1})_i=\big( \prox_{\frac{\gamma}{\omega_i}h_i} ( 2x^k-z_i^k-\gamma \nabla f(x^k)) \big)_i,~\varepsilon_i^k=v_i^{k+1}-u_i^{k+1}$, then $e_i^k = x^k-u_i^{k+1}=(z_i^k-z_i^{k+1})/{\lambda_k}+\varepsilon_i^k$, and ${g^k=\frac{1}{\gamma}x^{k} - \nabla f(x^{k}) - \frac{1}{\gamma}(\ssum_i\omega_iu_i^{k+1})}$. To save space, we only consider the case where $\lambda_k\in[\frac{1}{2\alpha},\frac{1}{\alpha}[$ is non-decreasing.

\begin{theorem}[Pointwise iteration-complexity bound of \eqref{eq:mnticln}]
\label{thm:pointwise_bnd_monotone_inclusion}
We have $g^k \in \sum_i \partial h_i(u_i^{k+1})$. Moreover, under the assumptions of Theorem \ref{thm:pointwise_irfpi_bounds},
\begin{equation*}
\norm{g^k+\nabla f(\ssum_i\omega_iu_i^{k+1})} \leq \frac{1}{\gamma}\sqrt{\frac{d_0^2 + C_2}{\tau_k(k+1)}}.
\end{equation*}
\end{theorem}

Let now $\bar{u}_i^k = \frac{1}{\Lambda_k}\sum_{j=0}^k \lambda_k u_i^{j+1}$, $\bar{x}^k = \frac{1}{\Lambda_k}\sum_{j=0}^k \lambda_j x^j$ and $\bar{g}^k=\frac{1}{\gamma}\bar{x}^{k} - \nabla f(\bar{x}^{k}) - \frac{1}{\gamma}(\ssum_i\omega_i\bar{u}_i^{k})$. We get the following.

\begin{theorem}[Ergodic iteration-complexity bound of \eqref{eq:mnticln}]
\label{thm:ergodic_bnd_monotone_inclusion}
Under the assumptions of Theorem~\ref{thm:ergodic_irfpi_bound}, we have
\begin{equation*}
\norm{\bar{g}^k+\nabla f(\ssum_i\omega_i\bar{u}_i^{k})} \leq {2(d_0+C_3)}/{(\gamma\Lambda_k)}.
\end{equation*}
\end{theorem}

\section{Numerical experiments}
\label{sct:num_expmt}

As an illustrative example, in this section, we consider the principal component pursuit (PCP) problem, and apply it to decompose a video sequence into its background and foreground components. The rationale behind this is that since the background is virtually the same in all frames, if the latter are stacked as columns of a matrix, it is likely to be low-rank (even of rank 1 for perfectly constant background). On the other hand, moving objects appear occasionally on each frame and occupy only a small fraction of it. Thus the corresponding component would be sparse.

Assume that a matrix real $M$ can be written as
\begin{equation*}
\label{eq:LR_Sparse}
M = X_{L,0} + X_{S,0} + N,
\end{equation*}
where a $X_{L,0}$ is low-rank, $X_{S,0}$ is sparse and $N$ is a perturbation matrix that accounts for model imperfection.
The PCP proposed in \cite{candes2011robust} attempts to provably recover $(X_{L,0},X_{S,0})$, to a good approximation, by solving a convex optimization. Here, toward an application to video decomposition, we also add a non-negativity constraint to the low-rank component, which leads to the convex problem
\[
\label{eq:rpca}
\min_{X_L, X_S}~ \tfrac{1}{2}\norm{M-X_L-X_S}_F^2+\mu_1\norm{X_S}_1 + \mu_2\norm{X_L}_\ast+\iota_{P_+}(X_L),
\]
where $\norm{\cdot}_F$ is the Frobenius norm, $\norm{\cdot}_\ast$ stands for the nuclear norm, and $\iota_{P_+}$ is the indicator function of the nonnegative orthant. 

One can observe that for fixed $X_L$, the minimizer of \eqref{eq:rpca} is $X_S^\star=\prox_{\mu_1{\norm{\cdot}_1}}(M - X_L)$. Thus, \eqref{eq:rpca} is equivalent to 
\[
\label{eq:rpcame}
\min_{X_L}~^1({\mu_1\norm{\cdot}_1})(M-X_L) + \mu_2\norm{X_L}_\star+\iota_{P_+}(X_L),
\]
where $^1({\mu_1\norm{\cdot}_1})(M-X_L) = \min_{Z} \frac{1}{2}\norm{M-X_L-Z}_F^2+\mu_1\norm{Z}_1$ is the Moreau Envelope of $\mu_1{\norm{\cdot}_1}$ of index 1. Since the Moreau envelope is differentiable with a 1-Lipschitz continuous gradient \cite{moreau1962decomposition}, \eqref{eq:rpcame} is a special instance of \eqref{eq:minPhi} and can be solved using Algorithm~\ref{alg:GFB}. Fig.~\ref{fig:frames} shows the recovered components for a video example. Fig.~\ref{fig:gx} displays the observed pointwise and ergodic rates and those predicted by Theorem~\ref{thm:pointwise_bnd_monotone_inclusion} and \ref{thm:ergodic_bnd_monotone_inclusion}.


\begin{figure}[!htb]
  \centering
  \includegraphics[width=0.495\linewidth]{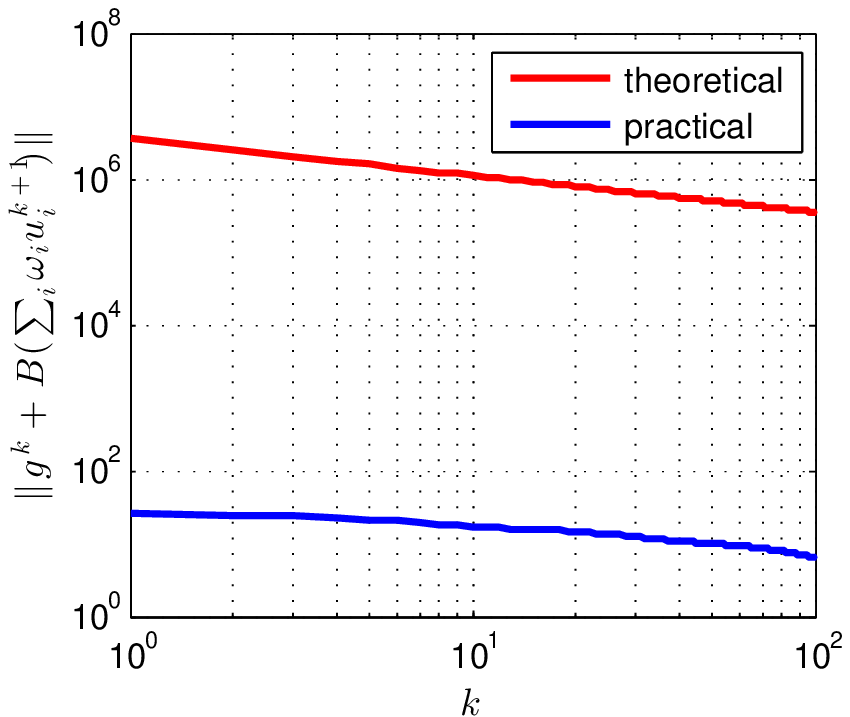}
  \includegraphics[width=0.495\linewidth]{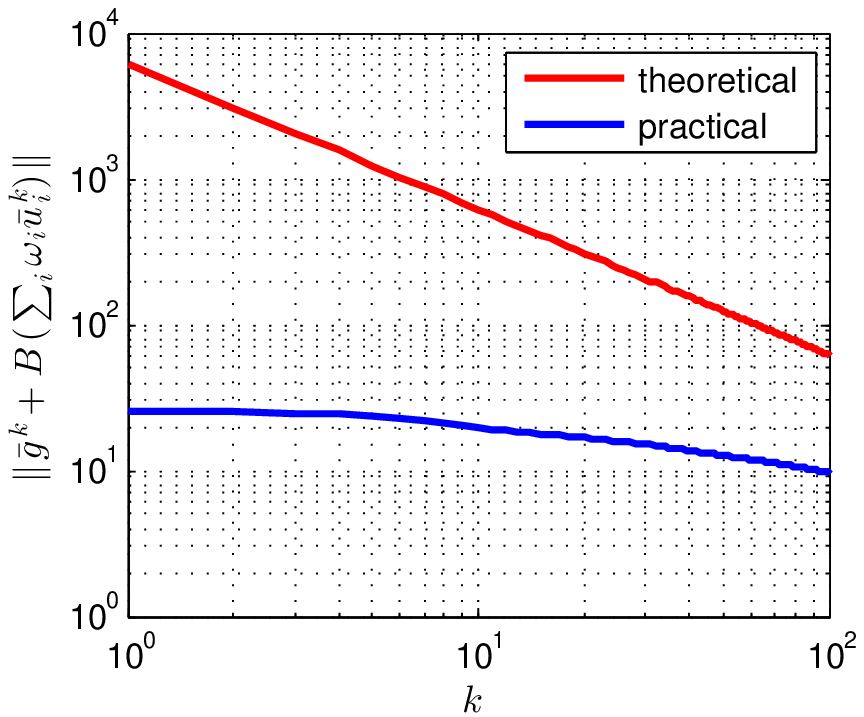}
\caption{Observed rates and theoretical bounds for the GFB applied to the PCP problem.}
\label{fig:gx}
\end{figure}

\begin{figure}[!htb]
\centering
\begin{tabular}{@{\hspace{0pt}}c@{\hspace{0pt}}c@{\hspace{0pt}}c}
\includegraphics[width=0.33\linewidth]{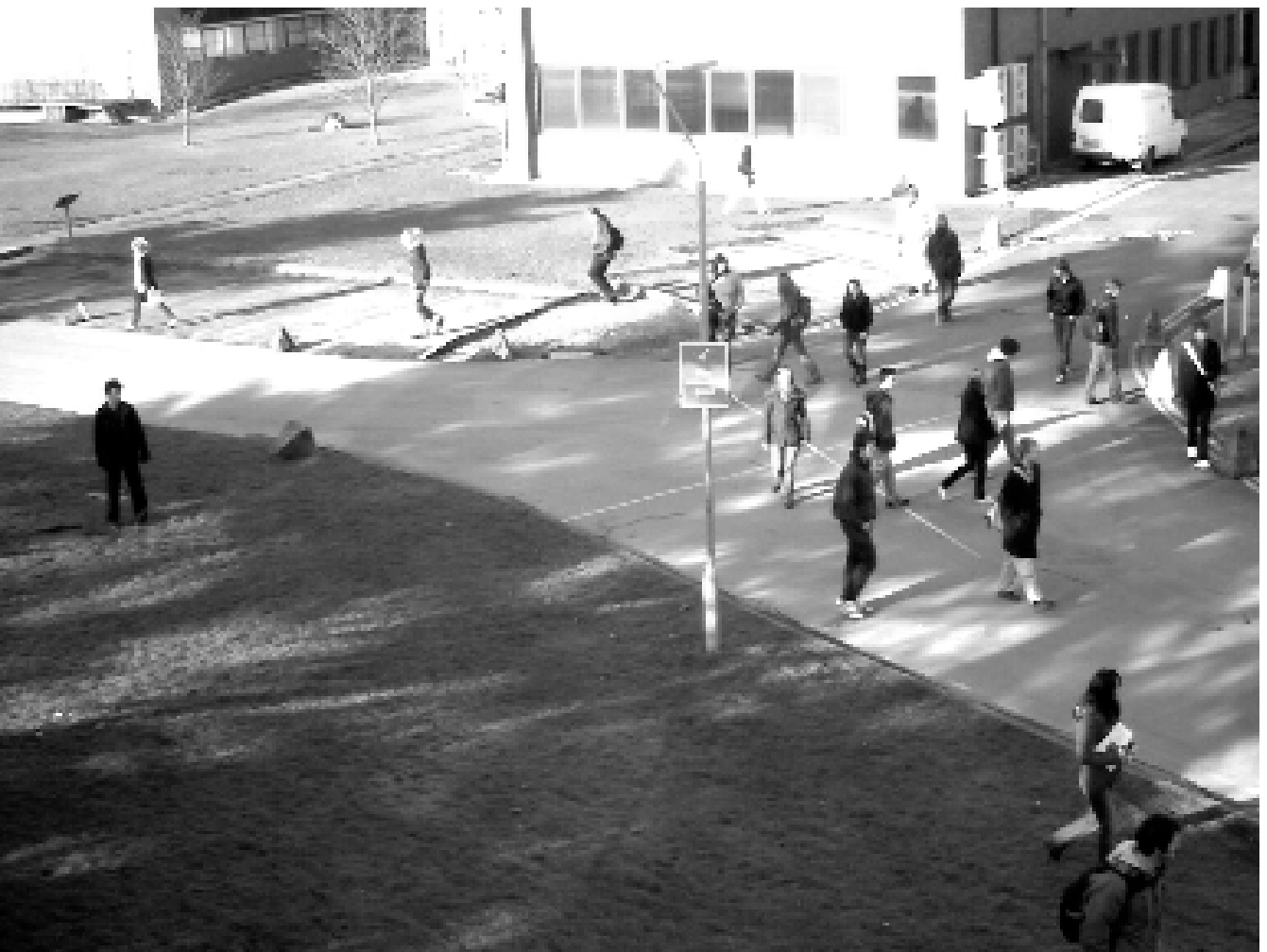} &
\includegraphics[width=0.33\linewidth]{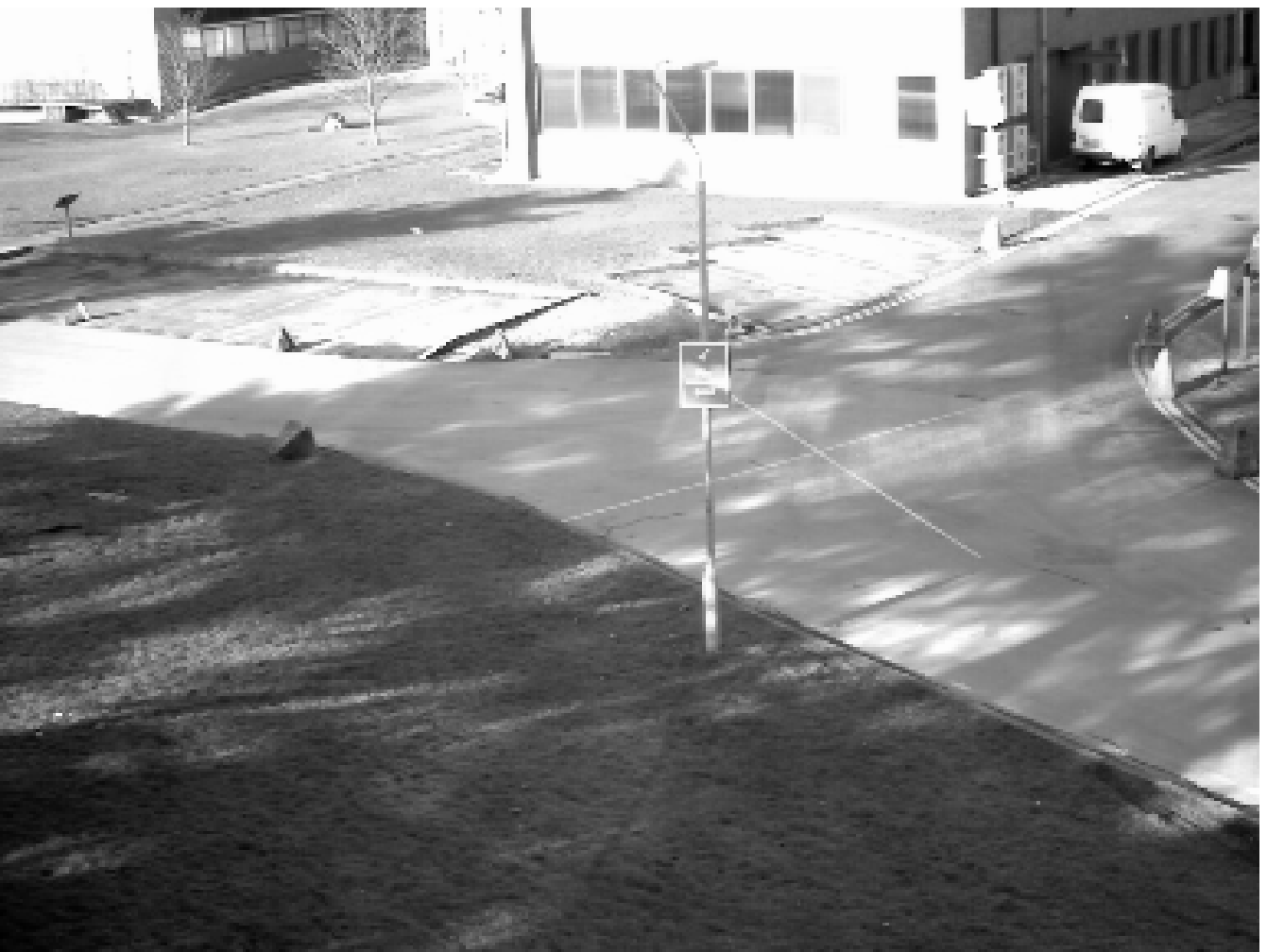} &
\includegraphics[width=0.33\linewidth]{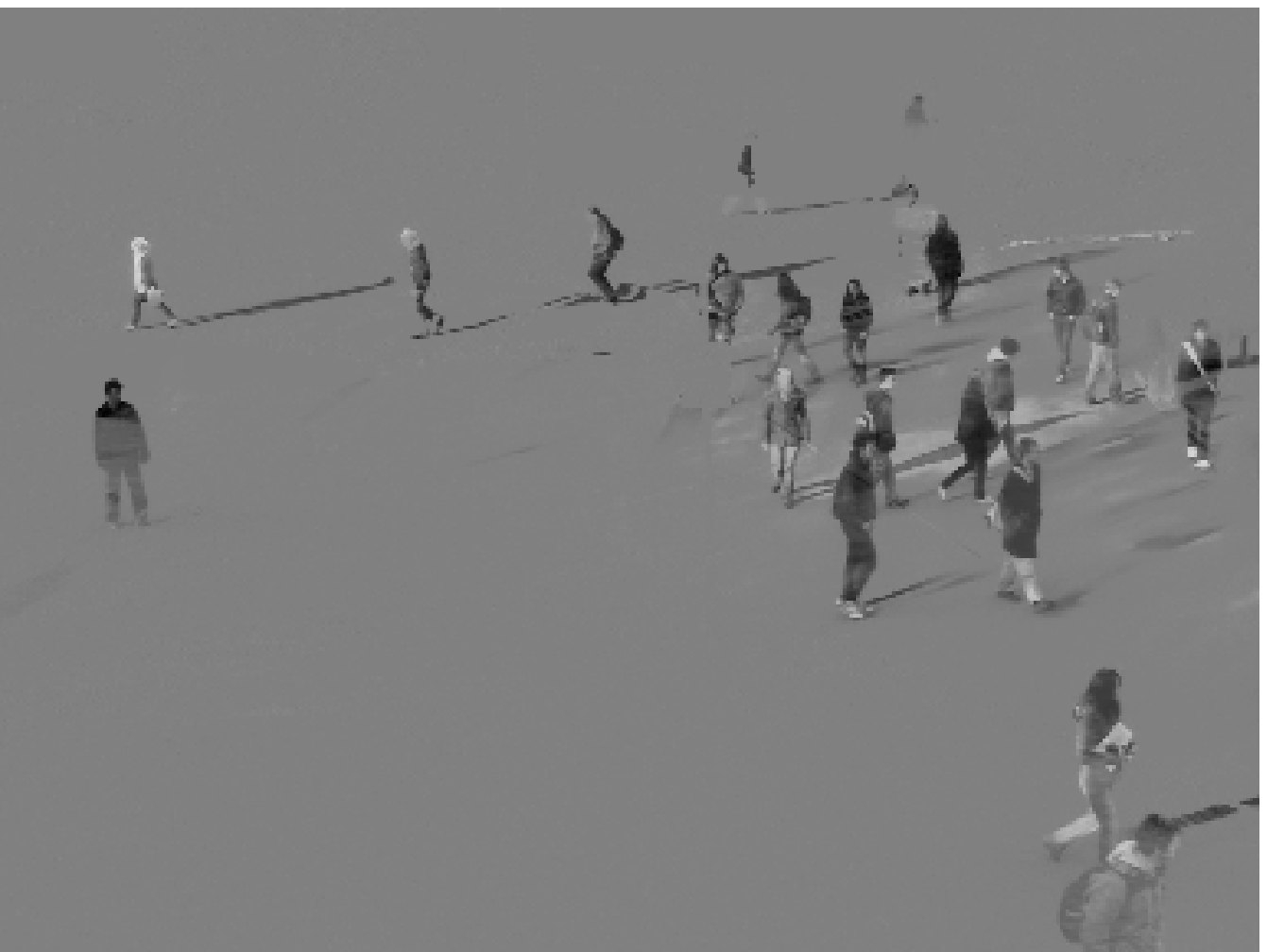} \\
\includegraphics[width=0.33\linewidth]{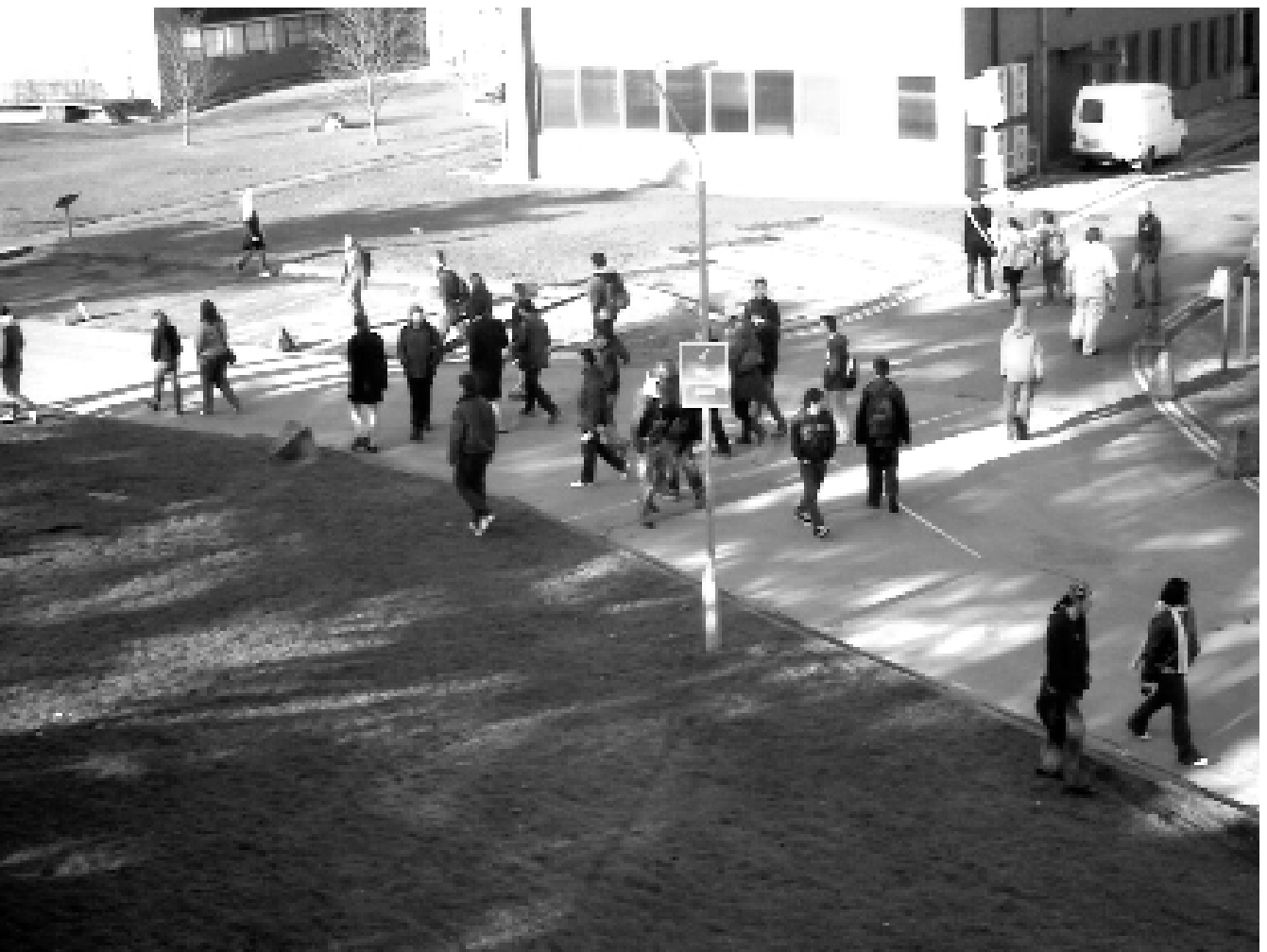} &
\includegraphics[width=0.33\linewidth]{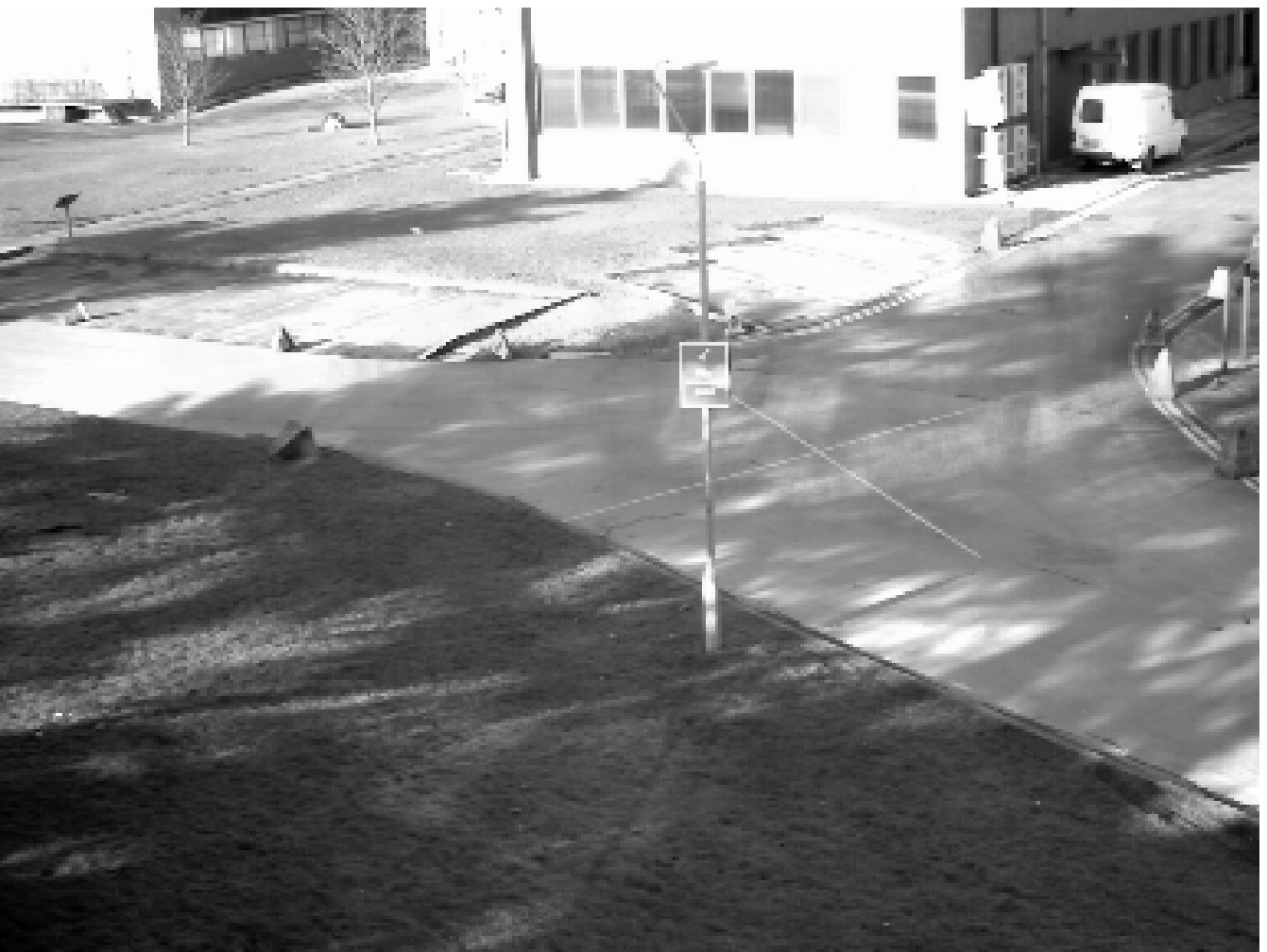} &
\includegraphics[width=0.33\linewidth]{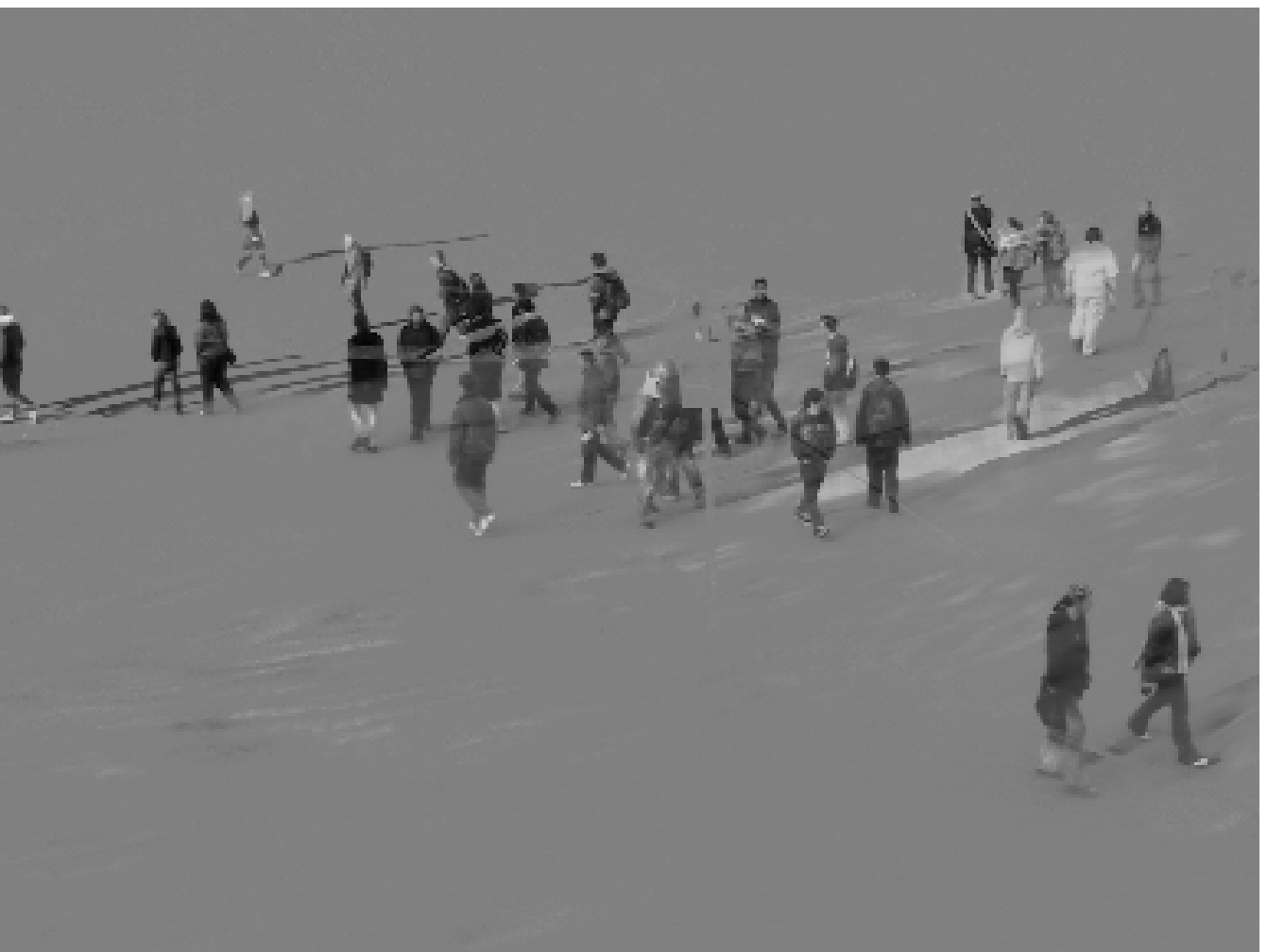} \\
\end{tabular}
\caption{Left: Original frames of a video (436 $384\times288$-pixel frames). Middle and Right: recovered background (low-rank) and foreground (sparse) components.}
\label{fig:frames}
\end{figure}

\section{Proofs}
\label{sec:proofs}

\subsection{Preparatory lemmata}
\begin{lemma}
\label{lemma:id_t}
Let $T\in\A(\alpha)$, then $\frac{1}{2\alpha}(\Id-T)$ is firmly non-expansive, \ie $\frac{1}{2\alpha}(\Id-T) \in \A(\frac{1}{2})$.
\end{lemma}
\begin{proof}
From the definition, it is straightforward to see that
\begin{equation*}
\tfrac{1}{2\alpha}(\Id-T) = \tfrac{1}{2}(\Id+(-R)).  \qedhere
\end{equation*}
\end{proof}

\begin{lemma}
\label{lemma:prop_subdiff}
If the operator $B$ is $\beta$-cocoercive, $\beta>0$, then
$\Id - \gamma B \in \A(\frac{\gamma}{2\beta}),~\mathrm{with}~\gamma\in]0,2\beta[$.
\end{lemma}
\begin{proof}
This is Baillon-Haddad theorem, see e.g., \cite[Proposition~4.33]{bauschke2011convex}.
\end{proof}

\begin{lemma}
For $\be^k$, the following inequality holds
\label{lemma:property_ek}
\begin{equation*}
\frac{1}{2\alpha\lambda_{k}}\bnorm{\be^k-\be^{k+1}}^2 \leq \bdprod{\be^k-\bepsilon^{k}}{\be^k-\be^{k+1}}.
\end{equation*}
\end{lemma}
\begin{proof}
Let $\bE: \bH \rightarrow \bH, \bz \mapsto (\bId-\bT)\bz$, then $\frac{1}{2\alpha}\bE=\frac{1}{2}\big(\bId+(-\bR)\big) \in \A(\frac{1}{2})$ (Lemma \ref{lemma:id_t}), and thus $\forall \bp,\bq \in \bH$,
\begin{equation*}
\label{eq:property_ek1}
\bnorm{\tfrac{1}{2\alpha}\bE(\bp) - \tfrac{1}{2\alpha}\bE(\bq)}^2 \leq \bdprod{\bp-\bq}{\tfrac{1}{2\alpha}\bE(\bp) - \tfrac{1}{2\alpha}\bE(\bq)},
\end{equation*}
substituting $\bz^{k}$ and $\bz^{k+1}$ for $\bp,~\bq$ yields the result.
\end{proof}


Denote $\nu_1 = 2\sup_{k\in\N}\bnorm{\bT_k\bz^{k}-\bz^\star} + \sup_{k\in\N}\lambda_k\bnorm{\bepsilon^{k}}$.
\begin{lemma}
\label{lemma:property_zk}
For $\bz^\star \in \Fix(\bT)$, $\lambda_k\in]0,\frac{1}{\alpha}[$, we have
\begin{equation*}
\bnorm{\bz^{k+1}-\bz^\star}^2 \leq \bnorm{\bz^{k}-\bz^\star}^2 - {\lambda_{k}(\frac{1}{\alpha}-\lambda_{k})}\bnorm{\be^{k}}^2 + \nu_1\lambda_k\bnorm{\bepsilon^{k}}.
\end{equation*}
\end{lemma}
\begin{proof}
Proposition~\ref{T_prop}(i) implies that $\bnorm{\bz^k - \bT\bz^k}^2 = \alpha^2\bnorm{\bz^k - \bR\bz^k}^2$ for some non-expansive operator $\bR$. Therefore, using \cite[Corollary~2.14]{bauschke2011convex} we get
\begin{align*}
\bnorm{\bz^{k+1}-\bz^\star}^2
\leq\ & \big( \bnorm{\bT_{k} \bz^{k}-\bz^\star} + \lambda_k\bnorm{\bepsilon^{k}} \big)^2 \\
\leq\ & (1-\alpha\lambda_{k})\bnorm{\bz^{k}-\bz^\star}^2 + \alpha\lambda_{k}\bnorm{\bR\bz^{k}-\bz^\star}^2 \\
~ &  - \alpha\lambda_k(1-\alpha\lambda_k)\bnorm{\bz^k-\bR\bz^k}^2 + \nu_1\lambda_k\bnorm{\bepsilon^{k}} \\
\leq\ & \bnorm{\bz^{k}-\bz^\star}^2 - \lambda_{k}(\frac{1}{\alpha}-\lambda_{k})\bnorm{\be^{k}}^2  + \nu_1\lambda_k\bnorm{\bepsilon^{k}}. 
\end{align*} 
\end{proof}


Denote $\nu_2 = 2\sup_{k\in\N}\bnorm{\be^k-\be^{k+1}}$.
\begin{lemma}
\label{lemma:property_ek2}
For $\lambda_k\in]0,\frac{1}{\alpha}[$, the sequence $(\be^k)_{k\in\N}$ obeys
\begin{equation*}
\bnorm{\be^{k+1}}^2 - \nu_2\bnorm{\bepsilon^k} \leq \bnorm{\be^k}^2.
\end{equation*}
\end{lemma}
\begin{proof}
Using Lemma~\ref{lemma:property_ek}, we get
\begin{equation*}
\begin{aligned}
\bnorm{\be^{k+1}}^2
=& \bnorm{\be^{k}}^2 - 2\bdprod{\be^{k}}{\be^{k}-\be^{k+1}} + \bnorm{\be^{k}-\be^{k+1}}^2 \\
\leq& \bnorm{\be^{k}}^2 - \frac{1-\alpha\lambda_k}{\alpha\lambda_k}\bnorm{\be^{k}-\be^{k+1}}^2 + \nu_2\bnorm{\bepsilon^k} \\
\leq& \bnorm{\be^{k}}^2 + \nu_2\bnorm{\bepsilon^k}. \qedhere
\end{aligned} 
\end{equation*}
\end{proof}

\subsection{Proofs of main results}
\paragraph*{Proof of Theorem~\ref{thm:pointwise_irfpi_bounds}}
\begin{enumerate}[label={\rm (\roman{*})}, ref={\rm (\roman{*})}, leftmargin=0cm,itemindent=0.5cm,labelwidth=\itemindent,labelsep=0cm,align=left]
\item This is an adaptation of \cite[Lemma 5.1]{combettes2004solving}.
\item
$(\mathrm{\rmnum{1}})$ implies that $\bnorm{\be^k}$ and $\bnorm{\bz^k-\bz^\star}$ are bounded, since $(\lambda_k\bnorm{\bepsilon^k})_{k\in\N} \in \ell_+^1$, therefore, $\nu_1$ and $\nu_2$ are bounded constant. Then from Lemma \ref{lemma:property_ek2}, $\forall j \leq k$,
\begin{equation*}
\bnorm{\be^k}^2 - \nu_2\msum_{\ell=j}^{k-1}\bnorm{\bepsilon^{\ell}} \leq \bnorm{\be^{j}}^2 ~.
\end{equation*}
Inserting this in Lemma \ref{lemma:property_zk}, and summing up over $j$, we get
\[
\label{eq:main_ineq}
\msum_{j=0}^k\tau_{j}\bnorm{\be^k}^2 \leq d_0^2 + \nu_1\msum_{j=0}^{k}\lambda_j\bnorm{\bepsilon^j} + \nu_2\msum_{j=0}^k\tau_{j}\msum_{\ell=j}^{k-1}\bnorm{\bepsilon^{\ell}},
\]
whence we obtain
\begin{equation*}
(k+1)\underline{\tau}\bnorm{\be^k}^2 \leq d_0^2 + \nu_1\msum_{j=0}^{k}\lambda_j\bnorm{\bepsilon^j} + \nu_2\overline{\tau}\msum_{\ell=0}^{k-1}(\ell+1)\bnorm{\bepsilon^{\ell}} ~.
\end{equation*}
Assumption~\ref{condition3} then yields the bound \eqref{eq:bound1}.
\item $\frac{1}{2\alpha} \leq \lambda_k \leq \sup_{k \in \N} \lambda_k < \frac{1}{\alpha}$ is non-decreasing implies that $\tau_k$ is non-increasing. Hence from \eqref{eq:main_ineq}, we get \eqref{eq:bound2}. 
\end{enumerate}

\paragraph*{Proof of Theorem~\ref{thm:ergodic_irfpi_bound}}
Nonexpansiveness of $\bT_{k}$ implies
\begin{equation*}
\bnorm{\bz^{k+1}-\bz^\star} \leq \bnorm{\bz^{0}-\bz^\star} + \msum_{j=0}^k \lambda_{j}\bnorm{\bepsilon^{j}} .
\end{equation*}
Combining this with the definition of $\bar{\be}^k$, we arrive at 
\begin{equation*}
\bnorm{\bar{\be}^{k}}\leq \frac{1}{\Lambda_k} ( \bnorm{\bz^{0}-\bz^{k+1}} + \msum_{j=0}^k\lambda_j\bnorm{\bepsilon^j} ) \leq {2(d_0 + C_3)}/{\Lambda_k}. 
\end{equation*}

\paragraph*{Proof of Theorem~\ref{thm:pointwise_bnd_monotone_inclusion}}
From the update formula of $u_i^{k+1}$, we get
\begin{align*}
{\gamma^{-1}}x^k-\nabla f(x^{k}) &- {\gamma^{-1}}(\ssum_i\omega_iu_i^{k+1}) + \gamma \nabla f(\ssum_i\omega_iu_i^{k+1}) \\
&~~~~~~ \in \msum_i\partial h_i(u_i^{k+1}) + \nabla f(\ssum_i\omega_iu_i^{k+1}) ~.
\end{align*}
Since $\Id -\gamma \nabla f \in \A(\frac{\gamma}{2\beta})$ (Lemma \ref{lemma:prop_subdiff}), hence nonexpansive, and using Theorem~\ref{pointwise_bnds3}, we arrive at
\begin{align*}
\norm{g^k+\nabla f(\ssum_i\omega_iu_i^{k+1})} 
\leq& {\gamma^{-1}}\norm{x^{k} - \msum_i\omega_iu_i^{k+1}}  \\
\leq& {\gamma^{-1}}\bnorm{\bx^k-\bu^{k+1}} \\
=& {\gamma^{-1}}\bnorm{\be^{k}} 
\leq \frac{1}{\gamma}\sqrt{\frac{d_0^2 + C_2}{\tau_k(k+1)}}. 
\end{align*}

\paragraph*{Proof of Theorem~\ref{thm:ergodic_bnd_monotone_inclusion}}
Owing to Theorem \ref{thm:pointwise_bnd_monotone_inclusion}, we get
\begin{align*}
\norm{\bar{g}^k+\nabla f(\ssum_i\omega_i\bar{u}_i^{k})}
\leq& {\gamma^{-1}}\norm{\bar{x}^{k} - \msum_i\omega_i\bar{u}_i^{k}}
\leq {\gamma^{-1}}\bnorm{\bar{\bx}^{k} - \bar{\bu}^{k}} \\
\leq& {\gamma^{-1}} \bnorm{\bar{\be}^{k}} \leq {2(d_0+C_3)}/{(\gamma\Lambda_k)}. 
\end{align*}


\bibliographystyle{IEEEbib}
\bibliography{egbib}

\end{document}